\documentclass{siamltex}
\usepackage{graphicx}
\usepackage{subfigure}
\usepackage{epsfig,psfig}
\usepackage{amsfonts,latexsym}
\usepackage{amsmath}

\newcommand{\R}{{\mathbb R}}

\newcommand{\bx}{\mbox{\boldmath{$x$}}}
\newcommand{\bb}{\mbox{\boldmath{$b$}}}

\newcommand{\br}{\mbox{\boldmath{$r$}}}

\newcommand{\be}{\mbox{\boldmath{$e$}}}

\newcommand{\bn}{\mbox{\boldmath{$n$}}}

\newcommand{\by}{\mbox{\boldmath{$y$}}}
\newcommand{\bz}{\mbox{\boldmath{$z$}}}
\newcommand{\sby}{\mbox{\boldmath{${\scriptstyle y}$}}}
\newcommand{\sbx}{\mbox{\boldmath{${\scriptstyle x}$}}}
\newcommand{\sbb}{\mbox{\boldmath{${\scriptstyle b}$}}}
\newcommand{\sbz}{\mbox{\boldmath{${\scriptstyle z}$}}}

\newcommand{\bzero}{\mbox{\boldmath{$0$}}}
\newcommand{\Null}[1]{\ensuremath{\mathcal{N}({#1})}}

\newcommand{\bl}{\begin{list}{ \ }{
\leftmargin=.325in}}
\newcommand{\el}{\end{list}}

\begin{document}
\date{}
\title{Regularization matrices determined \\ by matrix nearness problems}
\author{
Guangxin Huang\thanks{Geomathematics Key Laboratory of Sichuan, Chengdu University of
Technology, Chengdu, 610059, P. R. China. Email: {\tt huangx@cdut.edu.cn}. Research
supported by a grant from CSC, Research Fund Project (NS2014PY08) of SUSE and the young
scientific research backbone teachers of CDUT (KYGG201309).}
\and
Silvia Noschese\thanks{SAPIENZA Universit\`a di Roma, P.le A. Moro, 2,
I-00185 Roma, Italy. E-mail: {\tt noschese@mat.uniroma1.it}. Research
supported by a grant from SAPIENZA Universit\`a di Roma.}
\and
Lothar Reichel\thanks{Department of
Mathematical Sciences, Kent State University, Kent, OH 44242, USA.
E-mail: {\tt reichel@math.kent.edu}. Research supported in part by NSF
grant DMS-1115385.}
}

\maketitle

\begin{abstract}
This paper is concerned with the solution of large-scale linear discrete ill-posed
problems with error-contaminated data. Tikhonov regularization is a popular approach
to determine meaningful approximate solutions of such problems. The choice of
regularization matrix in Tikhonov regularization may significantly affect the quality of
the computed approximate solution. This matrix should be chosen to promote the recovery of
known important features of the desired solution, such as smoothness and monotonicity.
We describe a novel approach to determine regularization matrices with desired properties
by solving a matrix nearness problem. The constructed regularization matrix is the closest
matrix in the Frobenius norm with a prescribed null space to a given matrix. Numerical
examples illustrate the performance of the regularization matrices so obtained.
\end{abstract}

\begin{keywords}
Tikhonov regularization, regularization matrix,  matrix nearness problem.
\end{keywords}

\begin{AMS}
65R30, 65F22, 65F10.
\end{AMS}

\section{Introduction}
We are concerned with the computation of an approximate solution of linear
least-squares problems of the form
\begin{equation}\label{eq:sys}
   \min_{\sbx\in\R^n} \|K\bx-\bb\|, \qquad K \in \R^{m\times n},\quad \bb \in \R^m,
\end{equation}
with a large matrix $K$ with many singular values of different orders of magnitude close
to the origin. In particular, $K$ is severely ill-conditioned and may be singular. Linear
least-squares problems with a matrix of this kind often are referred to as linear discrete
ill-posed problems. They arise, for instance, from the discretization of linear ill-posed
problems, such as Fredholm integral equations of the first kind with a smooth kernel. The
vector $\bb$ of linear discrete ill-posed problems that arise in applications typically
represents measured data that is contaminated by an unknown error $\be\in\R^m$.

Let $\widehat{\bb}\in\R^m$ denote the unknown error-free vector associated with $\bb$,
i.e.,
\begin{equation}\label{noisefree}
\bb=\widehat{\bb}+\be,
\end{equation}
and let $\widehat{\bx}$ be the solution of the unavailable linear system of equations
\begin{equation}\label{consistent}
K\bx=\widehat{\bb},
\end{equation}
which we assume to be consistent. If $K$ is singular, then $\widehat{\bx}$ denotes the
solution of minimal Euclidean norm.

Let $K^\dag$ denote the Moore--Penrose pseudoinverse of $K$. The solution of minimal
Euclidean norm of \eqref{eq:sys}, given by
\[
K^\dag\bb=K^\dag\widehat{\bb}+K^\dag\be=\widehat{\bx}+K^\dag\be,
\]
typically is not a useful approximation of $\widehat{\bx}$ due to severe propagation of
the error $\be$. This depends on the large norm of $K^\dag$. Therefore, one generally
replaces the least-squares problem \eqref{eq:sys} by a nearby problem, whose solution is
less sensitive to the error $\be$. This replacement is known as regularization. One of the
most popular regularization methods is due to Tikhonov. This method replaces
\eqref{eq:sys} by a penalized least-squares problem of the form
\begin{equation}\label{eq:tik}
    \min_{{\sbx}\in\R^n}\left\{\|K\bx-\bb\|^2+\mu\|L\bx\|^2\right\},
\end{equation}
where $L\in \R^{p\times n}$ is referred to as a regularization matrix and the scalar
$\mu > 0$ as a regularization parameter; see, e.g., \cite{BRRS,EHN,Ha2}. Throughout this
paper $\|\cdot\|$ denotes the Euclidean vector norm or the spectral matrix norm. We
assume that the matrices $K$ and $L$ satisfy
\begin{equation}\label{nullcond}
\Null{K}\cap \Null{L} = \{\bzero\},
\end{equation}
where $\Null{M}$ denotes the null space of the matrix $M$. Then the minimization problem
\eqref{eq:tik} has the unique solution
\[
\bx_\mu=(K^TK+\mu L^TL)^{-1}K^T\bb
\]
for any $\mu>0$. The superscript $^T$ denotes transposition. When $L$ is the identity
matrix, the Tikhonov minimization problem (\ref{eq:tik}) is said to be in
\emph{standard form}, otherwise it is in \emph{general form}. We are interested in
minimization problems (\ref{eq:tik}) in general form.

The value of $\mu>0$ in (\ref{eq:tik}) determines how sensitive $\bx_\mu$ is to the error
$\be$, how close $\bx_\mu$ is to the desired solution $\widehat\bx$, and how
small the residual error $\bb-K\bx_\mu$ is. A suitable value of $\mu$ generally is not
explicitly known and has to be determined during the solution process.

Minimization problems \eqref{eq:tik} in general form with matrices $K$ and $L$ of small to
moderate size can be conveniently solved with the aid of the Generalized Singular Value
Decomposition (GSVD) of the matrix pair $\{K,L\}$; see, e.g., \cite{DNR2,Ha2}. We are interested
in developing solution methods for large-scale minimization problems \eqref{eq:tik}. These
problems have to be solved by an iterative method. However, the regularization matrices
$L$ derived also may be useful for problems of small size.

Common choices of regularization matrices $L$ in \eqref{eq:tik} when the least-squares
problem \eqref{eq:sys} is obtained by discretizing a Fredholm integral equation of the
first kind in one space-dimension are the $n\times n$ identity matrix $I_n$, and scaled
finite difference approximations of the first derivative operator,
\begin{equation}\label{L1}
L_1 = \frac{1}{2}\left[ \begin{array} {cccccc}
 1 &   -1    &     &    &   & \mbox{\Large 0} \\
   &  \phantom{-}1    &   -1    &     &    & \\
   &   &  \phantom{-}1    &  -1  &    &    \\
   &   &  & \ddots & \ddots  &   \\
 \mbox{\Large 0}   &        &        &    & \phantom{-}1  & -1
\end{array}
\right]\in{\R}^{(n-1)\times n},
\end{equation}
as well as of the second derivative operator,
\begin{equation}\label{L2}
L_2 = \frac{1}{4}\left[ \begin{array} {cccccc}
 -1   &  \phantom{-}2    &   -1    &     &    &\mbox{\Large 0} \\
      &   -1  &  \phantom{-}2    &  -1  &    &    \\
      &       & \ddots & \ddots & \ddots  &   \\
  \mbox{\Large 0}    &        &        &   -1  & \phantom{-}2      & -1
\end{array}
\right]\in{\R}^{(n-2)\times n}.
\end{equation}
The null spaces of these matrices are
\begin{equation}\label{nullL1}
{\mathcal N}(L_1)={\rm span}\{[1,1,\ldots,1]^T\}
\end{equation}
and
\begin{equation}\label{nullL2}
{\mathcal N}(L_2)={\rm span}\{[1,1,\ldots,1]^T,[1,2,\ldots,n]^T\}.
\end{equation}
The regularization matrices $L_1$ and $L_2$ damp fast oscillatory components of the
solution $\bx_\mu$ of \eqref{eq:tik} more than slowly oscillatory components. This can be
seen by comparing Fourier coefficients of the vectors $\bx$, $L_1\bx$, and $L_2\bx$; see,
e.g., \cite{RY}. These matrices therefore are referred to as \emph{smoothing
regularization matrices}. Here we think of the vector $\bx_\mu$ as a discretization of a
continuous real-valued function. The use of a smoothing regularization matrix can be
beneficial when the desired solution $\widehat\bx$ is a discretization of a smooth function.

The regularization matrix $L$ in (\ref{eq:tik}) should be chosen so that known important
features of the desired solution $\widehat{\bx}$ of \eqref{consistent} can be represented
by vectors in $\Null{L}$, because these vectors are not damped by $L$. For instance, if
the solution is known to be the discretization at equidistant points of a smooth
monotonically increasing function, then it may be appropriate to use the regularization
matrix \eqref{L2}, because its null space contains the discretization of linear functions.
Several approaches to construct regularization matrices with desirable properties are
described in the literature; see, e.g., \cite{CRS,DNR1,DR1,HJ,MRS,NR,RY}. Many of these
approaches are designed to yield square modifications of the matrices \eqref{L1} and
\eqref{L2} that can be applied in conjunction with iterative solution methods based on the
Arnoldi process. We will discuss the Arnoldi process more below.

The present paper describes a new approach to the construction of square regularization
matrices. It is based on determining the closest matrix with a prescribed null space to a
given square nonsingular matrix. For instance, the given matrix may be defined by
appending a suitable row to the finite difference matrix \eqref{L1} to make the matrix
nonsingular, and then prescribing a null space, say, \eqref{nullL1} or \eqref{nullL2}.
The distance between matrices is measured with the Frobenius norm,
\[
\|A\|_F:=\sqrt {\langle A,A\rangle},\qquad A\in{\R}^{p\times n},
\]
where the inner product between matrices is defined by
\[
\langle A,B\rangle:=\textrm{Trace}(B^TA),\qquad A,B\in{\R}^{p\times n}.
\]
Our reason for using the Frobenius norm is that the solution of the matrix nearness
problem considered in this paper can be determined with fairly little computations in
this setting.

We remark that commonly used regularization matrices in the literature, such as \eqref{L1}
and \eqref{L2}, are rectangular. Our interest in square regularization matrices stems from
the fact that they allow the solution of \eqref{eq:tik} by iterative methods that are based 
on the Arnoldi process.
Application of the Arnoldi process to the solution of Tikhonov minimization problems
\eqref{eq:tik} was first described in \cite{CMRS}; a recent survey is provided by Gazzola
et al. \cite{GNR}. We are interested in being able to use iterative solution methods that are
based on the
Arnoldi process because they only require the computation of matrix-vector products with
the matrix $A$ and, therefore, typically require fewer matrix-vector product evaluations
than methods that demand the computation of matrix-vector products with both $A$ and
$A^T$, such as methods based on Golub--Kahan bidiagonalization; see, e.g., \cite{LR} for
an example.

This paper is organized as follows. Section \ref{sec2} discusses matrix nearness problems
of interest in the construction of the regularization matrices. The application of
regularization matrices obtained by solving these nearness problems is discussed in
Section \ref{sec3}. Krylov subspace methods for the computation of an approximate solution
of \eqref{eq:tik}, and therefore of \eqref{eq:sys}, are reviewed in Section \ref{sec4}, 
and a few computed examples are presented in Section \ref{sec5}. Concluding remarks can be 
found in Section \ref{sec6}.

\section{Matrix nearness problems}\label{sec2}
This section investigates the distance of a matrix to the closest matrix with a
prescribed null space. For instance, we are interested in the distance of the
invertible square bidiagonal matrix
\begin{equation}\label{L1delta}
L_{1,\delta} = \frac{1}{2}\left[ \begin{array} {cccccc}
 1 &   -1    &     &    &  &  \mbox{\Large 0} \\
   &  \phantom{-}1    &   -1    &     &  &   \\
   &   &  \phantom{-}1    &  -1  &      &  \\
   &   &  & \ddots & \ddots  &   \\
   &        &        &    & \phantom{-}1 & -1  \\
 \mbox{\Large 0}   &        &        &    & \phantom{-}0 &\phantom{-}\delta
\end{array}
\right]\in{\R}^{n\times n}
\end{equation}
with $\delta>0$ to the closest matrix with the same null space as the rectangular matrix
(\ref{L1}). Regularization matrices of the form (\ref{L1delta}) with $\delta>0$ small have
been considered in \cite{CRS}; see also \cite{HJ} for a discussion.

Square regularization matrices have the advantage over rectangular ones that they can be
used together with iterative methods based on the Arnoldi process for Tikhonov
regularization \cite{CMRS,GNR} as well as in GMRES-type methods \cite{NRS}. These
applications have spurred the development of a variety of square regularization matrices.
For instance, it has been proposed in \cite{RY} that a zero row be appended to the
matrix (\ref{L1}) to obtain the square regularization matrix
\[
L_{1,0} = \frac{1}{2}\left[ \begin{array} {cccccc}
 1 &   -1    &     &    &   & \mbox{\Large 0} \\
   &  \phantom{-}1    &   -1    &     &    & \\
   &   &  \phantom{-}1    &  -1  &    &    \\
   &   &  & \ddots & \ddots  &   \\
   &        &        &    & \phantom{-}1  & -1 \\
 \mbox{\Large 0}   &        &        &    & \phantom{-}0 & \phantom{-}0
\end{array}
\right]\in{\R}^{n\times n}
\]
with the same null space. Among the questions that we are interested in is whether there
is a square regularization matrix that is closer to the matrix (\ref{L1delta}) than
$L_{1,0}$ and has the same null space as the latter matrix.
Throughout this paper ${\mathcal R}(A)$ denotes the range of the matrix $A$.

\begin{proposition}\label{prop1}
Let the matrix $V\in\R^{n\times\ell}$ have $1\leq\ell<n$ orthonormal columns and define
the subspace ${\mathcal V}:={\mathcal R}(V)$. Let ${\mathcal B}$ denote the subspace
of matrices $B\in\R^{p\times n}$ whose null space contains ${\mathcal V}$. Then $BV=0$ and
the matrix
\begin{equation}\label{k1}
\widehat{A}:=A(I_n-VV^T)
\end{equation}
satisfies the following properties:
\begin{enumerate}
\item
$\widehat{A}\in {\cal B}$;
\item
if $A\in {\cal B}$,  then $\widehat{A} \equiv A$;
\item
if $B\in {\cal B}$, then $\langle A-\widehat{A},B\rangle =0$.
\end{enumerate}
\end{proposition}

\begin{proof}
We have $\widehat{A}V=0$, which shows the first property. The second property
implies that $AV=0$, from which it follows that
\[
\widehat{A} = A-AVV^T=A.
\]
Finally, for any $B\in{\cal B}$,  we have
\begin{equation*}
\langle A-\widehat{A},B\rangle = \textrm{Trace}(B^TAVV^T) = 0,
\end{equation*}
where the last equality follows from the cyclic property of the trace.
\end{proof}

The following result is a consequence of Proposition \ref{prop1}.

\begin{corollary}\label{cor1}
The matrix (\ref{k1}) is the closest matrix to $A$ in ${\mathcal B}$ in the Frobenius
norm. The distance between the matrices $A$ and (\ref{k1}) is $\|AVV^T\|_F$.
\end{corollary}

The matrix closest to a given matrix with a prescribed null space also can be
characterized in a different manner that does not require an orthonormal basis of the
null space. It is sometimes convenient to use this characterization.

\begin{proposition}\label{prop3new}
Let  ${\mathcal B}$  be the subspace of matrices $B\in\R^{p\times n}$ whose null space
contains ${\mathcal R}(V)$, where $V\in\R^{n\times\ell}$ is a rank-$\ell$ matrix. Then
the closest matrix to $A$ in ${\mathcal B}$ in the Frobenius norm is $AP$, where
\begin{equation}\label{projector}
P=I_n-V\Omega^{-1}V^T
\end{equation}
with $\Omega=V^TV$.
\end{proposition}

\begin{proof}
Since the columns of $V$ are linearly independent, the matrix $\Omega$ is positive
definite and, hence, invertible. It follows that $P$ is an orthogonal projector with null
space ${\mathcal R}(V)$. The desired result now follows from Proposition \ref{prop1}.
\end{proof}

It follows from Proposition \ref{prop1} and Corollary \ref{cor1} with
$V=n^{-1/2}[1,1,\dots,1]^T$, or from Proposition \ref{prop3new}, that the closest matrix to
$L_{1,\delta}$ with null space ${\mathcal N}(L_1)$ is $L_{1,\delta}P$, where
$P=[P_{h,k}]\in\R^{n\times n}$ is the orthogonal projector given by
\[
P_{h,k}=\left\{\begin{array}{cc}
        \displaystyle{-\frac{1}{n}},\qquad & h \neq k, \\\\
        \displaystyle{\frac{n-1}{n}},\qquad & h=k. \\
	\end{array}\right.
\]
Hence,
\[
L_{1,\delta}P= \frac{1}{2}\left[ \begin{array} {cccccc}
 1 &   -1    &     &    &   & \mbox{\Large 0} \\
   &  \phantom{-}1    &   -1    &     &    & \\
   &   &  \phantom{-}1    &  -1  &    &    \\
   &   &  & \ddots & \ddots  &   \\
   &        &        &    & \phantom{-}1  & -1 \\
-\frac{\delta}{n}  & -\frac{\delta}{n}       &  \ldots       &\ldots     & -\frac{\delta}{n} & (1-\frac{1}{n})\delta
\end{array}
\right]\in{\R}^{n\times n}.
\]
Thus, $\|L_{1,\delta}-L_{1,\delta}P\|_F=\frac{\delta}{2\sqrt{n}}$ is smaller than
$\|L_{1,\delta}-L_{1,0}\|_F=\frac{\delta}{2}$.

We turn to square tridiagonal regularization matrices. The matrix
\[
L_{2,0} = \frac{1}{4}\left[ \begin{array} {cccccc}
 \phantom{-}0    & \phantom{-}0 & \phantom{-}0 & & & \mbox{\Large 0} \\
 -1   &  \phantom{-}2    &   -1    &     &    & \\
      &   -1  &  \phantom{-}2    &  -1  &    &    \\
      &       & \ddots & \ddots & \ddots  &   \\
      &        &        &   -1  & \phantom{-}2      & -1\\
 \mbox{\Large 0}   &        &        & \phantom{-}0  & \phantom{-}0 & \phantom{-}0
\end{array}
\right]\in{\R}^{n\times n}
\]
with the same null space as (\ref{L2}) is considered in \cite{DR1,RY}. We can apply
Propositions \ref{prop1} or \ref{prop3new} to determine whether this matrix is the closest
matrix to
\[
\widetilde{L}_2 = \frac{1}{4}\left[ \begin{array} {cccccc}
 \phantom{-}2    & -1 & & & & \mbox{\Large 0} \\
 -1   &  \phantom{-}2    &   -1    &     &    & \\
      &   -1  &  \phantom{-}2    &  -1  &    &    \\
      &       & \ddots & \ddots & \ddots  &   \\
      &        &        &   -1  & \phantom{-}2      & -1 \\
 \mbox{\Large 0}   &        &        &  & -1 & \phantom{-}2
\end{array}
\right]\in{\R}^{n\times n}
\]
with the null space (\ref{nullL2}). We also may apply Proposition \ref{prop4} below, which
is analogous to Proposition \ref{prop3new} in that that no orthonormal basis for the null
space is required. The result can be shown by direct computations.

\begin{proposition}\label{prop4}
Given $A\in \R^{p\times n}$, the closest matrix to $A$ in the Frobenius norm with a null space
containing the linearly independent vectors $v^{(1)},v^{(2)}\in \R^{n}$ is given by
$A(I_n-C)$, where
\begin{equation}
 C_{i,j}=\frac{\|v^{(1)}\|^2 v_i^{(2)}v_j^{(2)}-[v_i^{(2)}v_j^{(1)}+v_i^{(1)}v_j^{(2)}](v^{(1)},v^{(2)})+\|v^{(2)}\|^2 v_i^{(1)}v_j^{(1)}}
 {\|v^{(1)}\|^2\|v^{(2)}\|^2-(v^{(1)},v^{(2)})^2}. \label{gc}
\end{equation}
\end{proposition}

It follows easily from Proposition \ref{prop4} that the closest matrix to $\widetilde{L}_2$
with null space ${\mathcal N}(L_2)$ is $\widetilde{L}_2P$, where
$P=[P_{h,k}]\in\R^{n\times n}$ is an orthogonal projector defined by
\begin{equation}\label{P2}
P_{h,k}=\delta_{h,k}-\frac{2(n+1)(-3h+2n+1)+6k(2h-n-1)}{n(n+1)(n-1)},\quad h,k=1,\dots,n.
\end{equation}


The regularization matrices constructed above are generally nonsymmetric. We are also
interested in determining the distance between a given nonsingular symmetric matrix, such
as $\widetilde{L}_2$, and the closest symmetric matrix with a prescribed null space, such
as (\ref{nullL2}). The following results shed light on this.

\begin{proposition}\label{prop2}
Let the matrix $A\in\R^{n\times n}$ be symmetric, let $V\in\R^{n\times\ell}$ have
$1\leq\ell<n$ orthonormal columns and define the subspace
${\mathcal V}:={\mathcal R}(V)$. Let ${\mathcal B}_{sym}$ denote the subspace of
symmetric matrices $B\in\R^{n\times n}$ whose null space contains ${\mathcal V}$. Then
$BV=0$ and the matrix
\begin{equation}\label{k2}
\widehat{A}=(I_n-VV^T)A(I_n-VV^T)
\end{equation}
satisfies the following properties:
\begin{enumerate}
\item
$\widehat{A}\in {\cal B}_{sym}$;
\item
if $A\in {\cal B}_{sym}$,  then $\widehat{A} \equiv A$;
\item
if $B\in {\cal B}_{sym}$, then $\langle A-\widehat{A},B\rangle = 0$.
\end{enumerate}
\end{proposition}

\begin{proof}
We have $\widehat{A}=\widehat{A}^T$ and $\widehat{A}V=0$, which shows the first property.
The second property implies that $AV=V^TA=0$, from which it follows that
\[
\widehat{A} = A-VV^TA-AVV^T+VV^TAVV^T=A.
\]
Finally, for any $B\in{\cal B}_{sym}$, it follows from the cyclic property of the trace
that
\[
\langle A-\widehat{A},B\rangle =\textrm{Trace}(BVV^TA+BAVV^T-BVV^TAVV^T) = 0.
\]
\end{proof}

\begin{corollary}\label{cor2}
The matrix (\ref{k2}) is the closest matrix to $A$ in ${\cal B}_{sym}$ in the Frobenius
norm. The distance between the matrices $A$ and (\ref{k2}) is given by
$\|VV^TAVV^T-VV^TA-AVV^T\|_F$.
\end{corollary}

Proposition \ref{prop2} characterizes the closest matrix in ${\mathcal B}_{sym}$ to a
given symmetric matrix $A\in\R^{n\times n}$. The following proposition provides another
characterization that does not explicitly use an orthonormal basis for the prescribed null
space. The result follows from Proposition \ref{prop2} and Corollary \ref{cor2} in a
straightforward manner.

\begin{proposition}\label{cor4new}
Let  ${\mathcal B}_{sym}$  be the subspace of symmetric matrices $B\in\R^{n\times n}$
whose null space contains ${\mathcal R}(V)$, where $V\in\R^{n\times\ell}$ is a rank-$\ell$
matrix. Then the closest matrix to the symmetric matrix $A$ in ${\mathcal B}_{sym}$ in the
Frobenius norm is $PAP$, where $P\in\R^{n\times n}$ is defined by (\ref{projector}).
\end{proposition}

We are interested in determining the closest symmetric matrix to $\widetilde{L}_2$ with
null space in (\ref{nullL2}). It is given by $P\widetilde{L}_2P$, with $P$ defined in
(\ref{P2}). The inequalities
\[
\|\widetilde{L}_2-\widetilde{L}_2P\|_F<\|\widetilde{L}_2-P\widetilde{L}_2P\|_F <
\|\widetilde{L}_2-L_{2,0}\|_F=\frac{\sqrt{10}}{4}
\]
are easy to show. Figure \ref{fig2.1} displays the three 
distances for increasing matrix dimensions.

\begin{figure*}
\begin{center}
\includegraphics[scale=0.35]{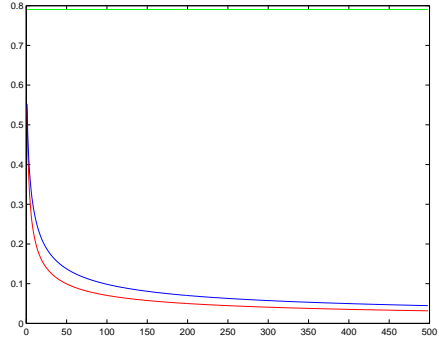} \\
\end{center}
\caption{Distances $\|\widetilde{L}_2-L_{2,0}\|_F$ (dashed  curve), 
$\|\widetilde{L}_2-P\widetilde{L}_2P\|_F$ (dash-dotted curve), and  
$\|\widetilde{L}_2-\widetilde{L}_2P\|_F$ (solid curve) as a function of the matrix
order $n$.}\label{fig2.1}
\end{figure*}

\section{Application of the regularization matrices}\label{sec3}
In this section we discuss the use of regularization matrices of the form $L=\tilde{L}P$ and $L=P\tilde{L}P$
in the Tikhonov minimization problem (\ref{eq:tik}), where $P$ is an orthogonal projector
and $\tilde{L}$ is nonsingular. We solve the problem (\ref{eq:tik}) by transforming it to
standard form in two steps. First, we let $\by=P\bx$ and then set $\bz=\tilde{L}\by$. 
Following Eld\'en \cite{El} or Morigi et al. \cite{MRS}, we express the Tikhonov
minimization problem
\begin{equation}\label{eq:tik1}
    \min_{{\sbx}\in\R^n}\left\{\|K\bx-\bb\|^2+\mu\|\tilde{L}P\bx\|^2\right\}
\end{equation}
in the form
\begin{equation}\label{eq:tik2}
    \min_{{\sby}\in\R^n}\left\{\|K_1\by-\bb_1\|^2+\mu\|\tilde{L}\by\|^2\right\},
\end{equation}
where
\[
K_1=KP^{\dag}_K,\qquad P^{\dag}_K=(I_n-(K(I_n-P^\dag P))^\dag K)P
\]
and
\[
\bb_1=\bb-K\bx^{(0)},\qquad \bx^{(0)}=(K(I_n-P^\dag P))^{\dag}\bb.
\]

Let the columns of $V_\ell$ form an orthonormal basis for the desired null space
of $L$. Then $P=I_n-V_\ell V_\ell^T$. Determine the QR factorization
\begin{equation}\label{qrfact}
    K V_\ell=QR,
\end{equation}
where $Q\in\R^{m\times\ell}$ has orthonormal columns and $R\in\R^{\ell\times\ell}$ is
upper triangular. It follows from (\ref{nullcond}) that $R$ is nonsingular, and we obtain
\begin{equation}\label{PK}
P_K^\dag = I_n-V_\ell R^{-1} Q^TK,\qquad KP_K^\dag = (I_m-QQ^T)K.
\end{equation}
These formulas are convenient to use in iterative methods for the solution of
\eqref{eq:tik2}; see \cite{MRS} for details. Let $\by_{\mu}$ solve \eqref{eq:tik2}. Then
the solution of (\ref{eq:tik1}) is given by $\bx_{\mu}=P_K^{\dag}\by_{\mu}+\bx^{(0)}$.

We turn to the solution of (\ref{eq:tik2}). This minimization problem can be expressed
in standard form
\begin{equation}\label{eq:tik3}
    \min_{{\sbz}\in\R^n}\left\{\|K_2\bz-\bb_1\|^2+\mu\|\bz\|^2\right\},
\end{equation}
where $K_2=K_1\tilde{L}^{-1}$. Let $\bz_{\mu}$ solve (\ref{eq:tik3}). Then the solution of
(\ref{eq:tik2}) is given by $\by_{\mu}=\tilde{L}^{-1}\bz_{\mu}$. In actual computations,
we evaluate $\tilde{L}^{-1}\bz$ by solving a linear system of equations with $\tilde{L}$.
We can similarly solve the problem (\ref{eq:tik}) with $L=P\tilde{L}P$ by transforming it to
standard form in three steps, where the first two steps are the same as above and the last
step is similar to the first step of the case with $L=\tilde{L}P$.

It is desirable that the matrix $\tilde{L}$ not be very ill-conditioned
to avoid severe error propagation when solving linear systems of equations with this
matrix. For instance, the condition number of the regularization matrix
$L_{1,\delta}$, defined by \eqref{L1delta}, depends on the parameter $\delta>0$. Clearly,
the condition number of $L_{1,\delta}$, defined as the ratio of the largest and smallest
singular value of the matrix, is large for $\delta>0$ ``tiny'' and of moderate size
for $\delta=1$. In the computations reported in Section \ref{sec5}, we use the latter
value.

\section{Krylov subspace methods and the discrepancy principle}\label{sec4}
A variety of Krylov subspace iterative methods are available for the solution of the
Tikhonov minimization problem \eqref{eq:tik3}; see, e.g., \cite{CMRS,CR,GNR,NRS} for
discussions and references. The discrepancy principle is a popular approach to determining
the regularization parameter $\mu$ when a bound $\varepsilon$ for the norm of the error
$\be$ in $\bb$ is known, i.e., $\|\be\|\leq\varepsilon$. It can be shown that the error in
$\bb_1$ satisfies the same bound. The discrepancy principle prescribes that $\mu>0$ be
chosen so that the solution $\bz_\mu$ of \eqref{eq:tik3} satisfies
\begin{equation}\label{nonlin}
\|K_2\bz_\mu-\bb_1\|=\eta\varepsilon,
\end{equation}
where $\eta>1$ is a constant independent of $\varepsilon$. This is a nonlinear
equation of $\mu$.

We can determine an approximation of $\bz_\mu$ by applying an iterative method to the
linear system of equations
\begin{equation}\label{K2sys}
K_{2}\bz=\bb_{1}
\end{equation}
and terminating the iterations sufficiently early. This is simpler than solving
\eqref{eq:tik3}, because it circumvents the need to solve the nonlinear equation
\eqref{nonlin} for $\mu$. We therefore use this approach in the computed examples of
Section \ref{sec5}. Specifically, we apply the Range Restricted GMRES (RRGMRES) iterative
method described in \cite{NRS}. At the $k$th step, this method computes an approximate
solution $\bz_k$ of \eqref{K2sys} as the solution of the minimization problem
\[
\min_{\sbz\in{\mathcal K}_k(K_{2},K_{2}\sbb_{1})} \|K_{2}\bz-\bb_{1}\|,
\]
where ${\mathcal K}_k(K_{2},K_2\bb_{1}):=
\mbox{span}\{K_{2}\bb_{1},K_{2}^2\bb_{1},\ldots,K_{2}^k\bb_{1}\}$ is a Krylov subspace.
The discrepancy principle prescribes that the iterations with RRGMRES be terminated as
soon as an iterate $\bz_k$ that satisfies
\begin{equation}\label{discrp1}
\|K_2\bz_k-\bb_1\|\leq\eta\varepsilon
\end{equation}
has been computed. The number of iterations required to satisfy this stopping criterion
generally increases as $\varepsilon$ is decreased. Using the transformation from
$\bz_\mu$ to $\bx_\mu$ described in Section \ref{sec3}, we transform $\bz_k$ to an
approximate solution $\bx_k$ of \eqref{eq:sys}. Further details can be found in
\cite{NRS}. Here we only note that $\|K_2\bz_k-\bb_1\|$ can be computed without
explicitly evaluating the matrix-vector product $K_2\bz_k$.

\section{Numerical examples}\label{sec5}
We illustrate the performance of regularization matrices of the form $L=\tilde{L}P$ and $L=P\tilde{L}P$.
The error vector ${\be}$ has in all examples normally distributed pseudorandom
entries with mean zero, and is normalized to correspond to a chosen noise level
\[
\nu:=\frac{\|{\be}\|}{\|\widehat{\bb}\|},
\]
where $\widehat{\bb}$ denotes the error-free right-hand side vector in (\ref{consistent}).
We let $\eta=1.01$ in (\ref{discrp1}) in all examples. Throughout this section $P_1$ and
$P_2$ denote orthogonal projectors with null spaces \eqref{nullL1} and \eqref{nullL2},
respectively. All computations are carried out on a computer with an Intel Core
i5-3230M @ 2.60GHz processor and 8GB ram using MATLAB R2012a. The computations are
done with about $15$ significant decimal digits.

\begin{table}[tbh]
\begin{center}
\begin{tabular}{cccc}
 \hline
reg. mat. & \# iterations $k$ & \# mat.-vec. prod. &
$\|{\bx}_k-\widehat{\bx}\|/\|\widehat{\bx}\|$\\
\hline
&&noise level $\nu=1\cdot10^{-2}$&\\
\hline
$I$                  & $4$ & $5$ & $3.5 \cdot 10^{-2}$ \\
$L_{1,0}$            & $3$ & $4$ & $6.5 \cdot 10^{-3}$ \\
$L_{1,\delta}P_{1}$  & $5$ & $7$ & $5.1 \cdot 10^{-3}$ \\
$L_{2,0}$            & $3$ & $4$ & $6.6 \cdot 10^{-3}$ \\
$\tilde{L}_{2}P_{2}$ & $4$ & $7$ & $9.5 \cdot 10^{-3}$ \\
$P_{2}\tilde{L}_{2}P_{2}$ & $1$ & $7$ & $1.5 \cdot 10^{-2}$ \\
\hline
&&noise level $\nu=1\cdot10^{-3}$&\\
\hline
$I$                  & $9$ & $10$ & $1.7 \cdot 10^{-2}$ \\
$L_{1,0}$            & $3$ & $4$ & $4.5 \cdot 10^{-3}$ \\
$L_{1,\delta}P_{1}$  & $7$ & $9$ & $1.2 \cdot 10^{-3}$ \\
$\tilde{L}_{2,0}$    & $3$ & $4$ & $4.5 \cdot 10^{-3}$ \\
$\tilde{L}_{2}P_{2}$ & $5$ & $8$ & $4.1 \cdot 10^{-3}$ \\
$P_{2}\tilde{L}_{2}P_{2}$ & $5$ & $11$ & $1.4 \cdot 10^{-2}$ \\
 \hline
 &&noise level $\nu=1\cdot10^{-4}$&\\
\hline
$I$                  & $10$ & $11$ & $6.1 \cdot 10^{-3}$ \\
$L_{1,0}$            & $6$ & $7$ & $2.8 \cdot 10^{-3}$ \\
$L_{1,\delta}P_{1}$  & $9$ & $11$ & $2.0 \cdot 10^{-3}$ \\
$\tilde{L}_{2,0}$    & $6$ & $7$ & $2.8 \cdot 10^{-3}$ \\
$\tilde{L}_{2}P_{2}$ & $7$ & $10$ & $2.1 \cdot 10^{-3}$ \\
$P_{2}\tilde{L}_{2}P_{2}$ & $6$ & $12$ & $3.9 \cdot 10^{-3}$ \\
 \hline
\end{tabular}
\end{center}
\caption{Example 5.1: Number of iterations, number of matrix-vector product evaluations
with the matrix $K$, and relative error in approximate solutions ${\bx}_k$ determined by
truncated iteration with RRGMRES using the discrepancy principle and different
regularization matrices for several noise levels.}\label{tab5.1}
\end{table}

\begin{figure*}[h!tbp]
\centering
\subfigure[]{\includegraphics[height=4cm,width=6cm]{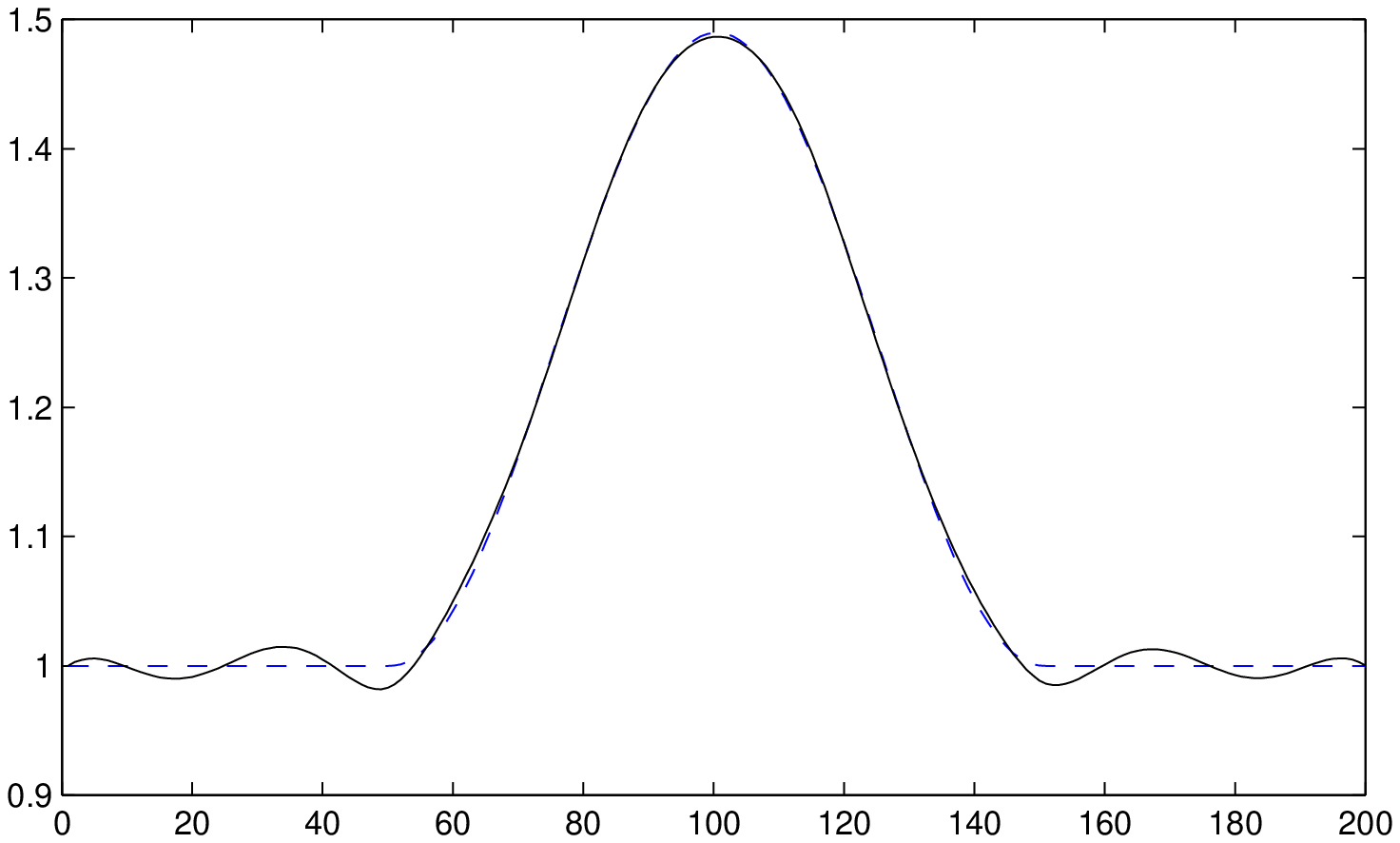}}
\subfigure[]{\includegraphics[height=4cm,width=6cm]{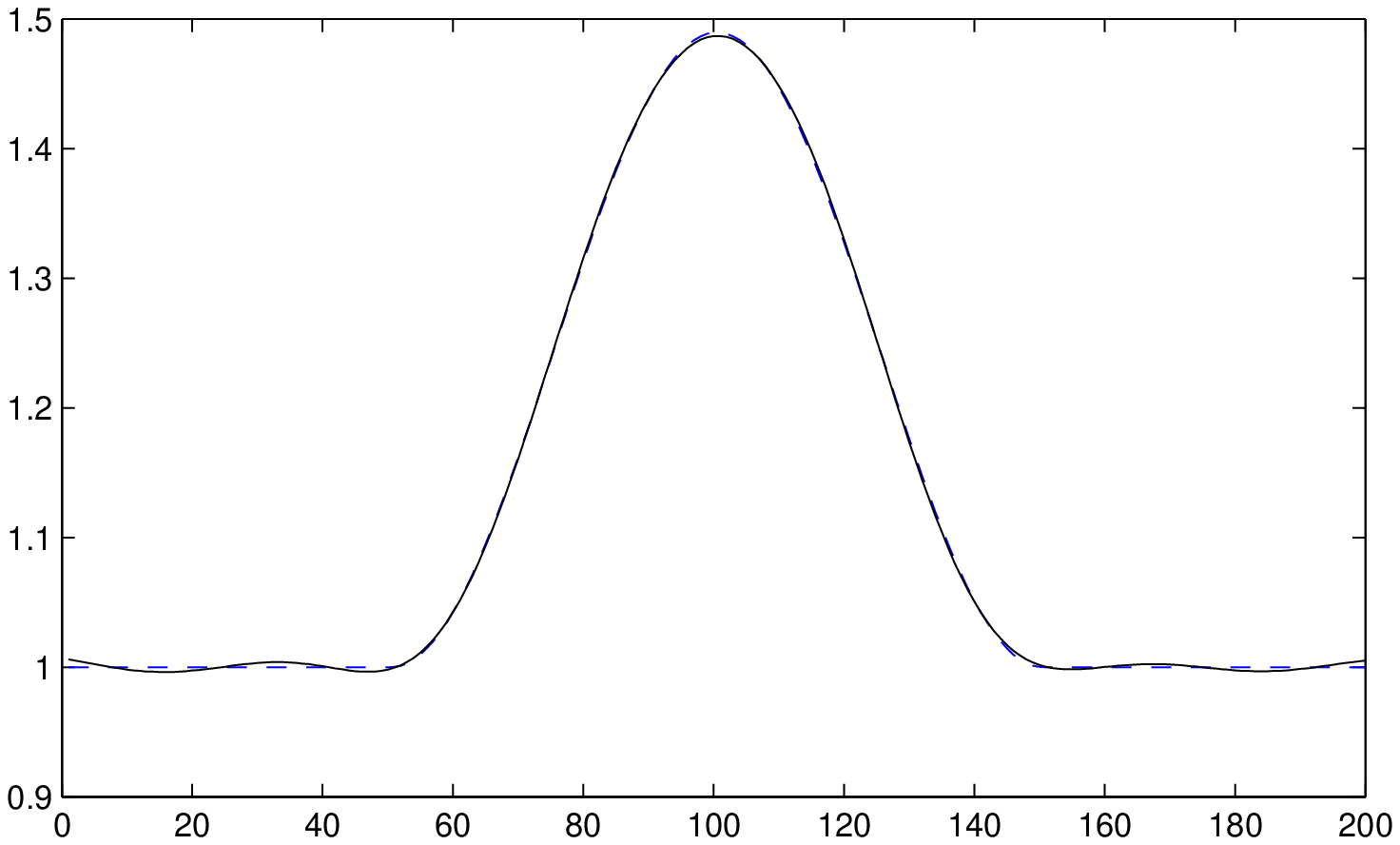}}
\subfigure[]{\includegraphics[height=4cm,width=6cm]{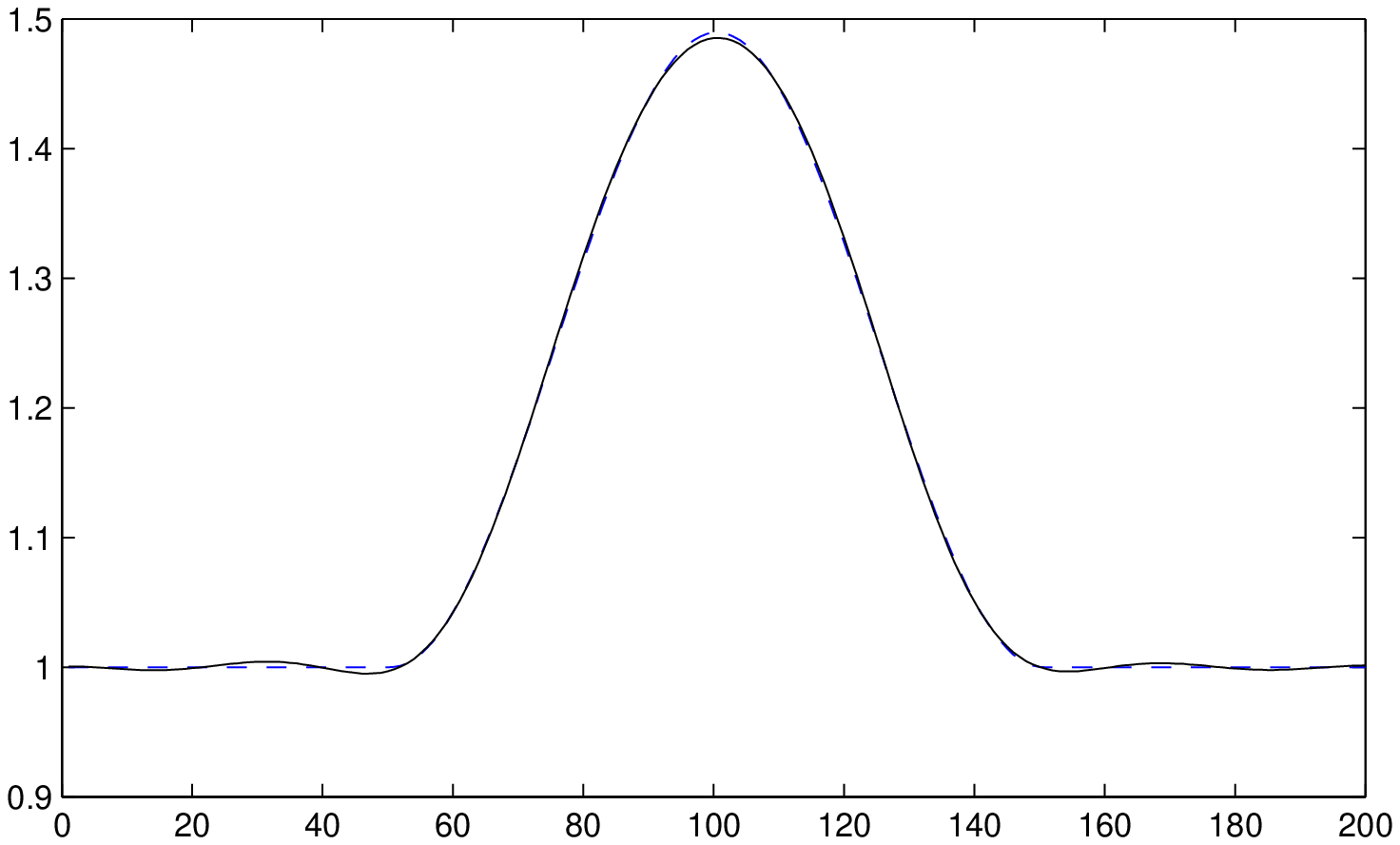}}
\subfigure[]{\includegraphics[height=4cm,width=6cm]{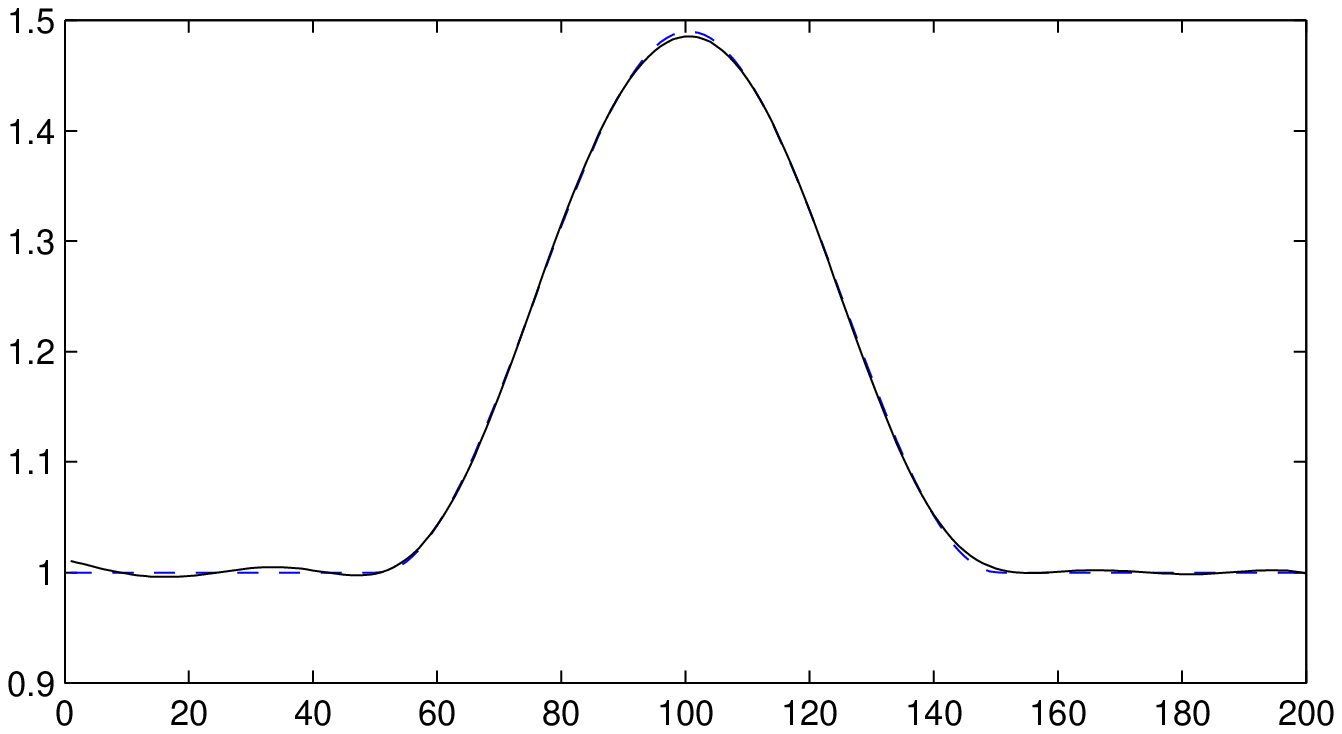}}
\caption{Example 5.1: Continuous curves: Computed approximate solutions ${\bx}_k$
determined by truncated iteration with RRGMRES using the discrepancy principle. The noise
level is $\nu=1\cdot10^{-4}$. (a) Iterate ${\bx}_{10}$ determined without regularization
matrix ($L=I$), (b) iterate ${\bx}_{9}$ determined with the regularization matrix
$L=L_{1,\delta}P_{1}$, (c) iterate ${\bx}_{7}$ determined with the regularization
matrix $L=\tilde{L}_{2}P_{2}$, and (d) iterate ${\bx}_{6}$ determined with the regularization
matrix $L=P_{2} \tilde{L}_{2}P_{2}$. The dashed curves show the desired solution $\widehat{\bx}$.}
\label{fig5.1}
\end{figure*}

Example 5.1. Consider the Fredholm integral equation of the first kind,
\begin{equation}\label{phillips}
\int_{-6}^{6}\kappa(\tau,\sigma)x(\sigma)d\sigma=g(\tau), \qquad
-6\le \tau \le 6,
\end{equation}
with kernel and solution given by
\[
\kappa(\tau,\sigma):=x(\tau-\sigma)
\]
and
\[
x(\sigma):=\left\{
\begin{array}{cl} 1 + \cos(\frac{\pi}{3}\sigma),\qquad &  \mbox{~~if~~} |\sigma|<3, \\
   0,\qquad &  \mbox{~~otherwise}.
   \end{array}\right.
\]
This equation is discussed by Phillips \cite{Ph}. We use the MATLAB code {\sf phillips}
from \cite{Ha3} to discretize (\ref{phillips}) by a Galerkin method with $200$ orthonormal
box functions as test and trial functions. The code produces the matrix
$K\in{\R}^{200\times 200}$ and a scaled discrete approximation of $x(\sigma)$. Adding
${\bn}_1=[1,1,\ldots,1]^{T}$ to the latter yields the vector $\widehat{\bx}\in{\R}^{200}$
with which we compute the error free right-hand side $\widehat{\bb}:=K\widehat{\bx}$. This
provides an example of a problem for which it is undesirable to damp the ${\bn}_1$-component
in the computed approximation of $\widehat{\bx}$.

The error vector ${\be}\in{\R}^{200}$ is generated as described above and normalized to
correspond to different noise levels
$\nu\in\{1\cdot 10^{-2},1\cdot 10^{-3},1\cdot 10^{-4}\}$. The data vector ${\bb}$ in
(\ref{eq:sys}) is obtained from (\ref{noisefree}).

Table \ref{tab5.1} displays results obtained with RRGMRES for several regularization
matrices and different noise levels, and Figure \ref{fig5.1} shows three computed
approximate solutions obtained for the noise level $\nu=1\cdot 10^{-4}$. The iterations
are terminated by the discrepancy principle (\ref{discrp1}). From Table \ref{tab5.1} and
Figure \ref{fig5.1}, we can see that the regularization matrix  $L=L_{1,\delta}P_{1}$
yields the best approximation of $\widehat{\bx}$. The worst approximation is obtained when
no regularization matrix is used with RRGMRES. This situation is denoted by $L=I$.
Table \ref{tab5.1} shows both the number of iterations and the number of matrix-vector
product evaluations
with the matrix $K$. The fact that the latter number is larger depends on the $\ell$
matrix-vector product evaluations with $K$ required to evaluate the left-hand side of
(\ref{qrfact}) and the matrix-vector product with $K$ needed for evaluating the
product of $P_K^\dag$ with a vector; cf. (\ref{PK}). $\Box$

\begin{table}[tbh]
\begin{center}
\begin{tabular}{cccc}
 \hline
reg. mat. & \# iterations $k$ & \# mat.-vec. prod. &
$\|{\bx}_k-\widehat{\bx}\|/\|\widehat{\bx}\|$\\
 \hline
&&noise level $v=1\cdot10^{-2}$&\\
\hline
$I$                    & $4$ & $5$ & $2.4 \cdot 10^{-1}$ \\
$L_{1,0}$              & $1$ & $2$ & $1.1 \cdot 10^{-2}$ \\
$L_{1,\delta}P_{1}$    & $1$ & $3$ & $2.6 \cdot 10^{-2}$ \\
$\tilde{L}_{2,0}$      & $0$ & $1$ & $3.1 \cdot 10^{-3}$ \\
$\tilde{L}_{2}P_{2}$   & $0$ & $3$ & $3.1 \cdot 10^{-3}$ \\
$P_{2}\tilde{L}_{2}P_{2}$   & $0$ & $6$ & $1.1 \cdot 10^{-2}$ \\
\hline
&&noise level $v=1\cdot10^{-3}$&\\
\hline
$I$                   & $8$ & $9$ & $1.5 \cdot 10^{-1}$ \\
$L_{1,0}$              & $3$ & $4$ & $7.3 \cdot 10^{-3}$ \\
$L_{1,\delta}P_{1}$   & $7$ & $9$ & $4.6 \cdot 10^{-2}$ \\
$\tilde{L}_{2,0}$      & $1$ & $2$ & $1.9 \cdot 10^{-3}$ \\
$\tilde{L}_{2}P_{2}$  & $1$ & $4$ & $1.7 \cdot 10^{-3}$ \\
$P_{2}\tilde{L}_{2}P_{2}$   & $0$ & $6$ & $1.1 \cdot 10^{-3}$ \\
 \hline
 &&noise level $v=1\cdot10^{-4}$&\\
\hline
$I$                    & $13$ & $14$ & $1.0 \cdot 10^{-1}$ \\
$L_{1,0}$              & $2$ & $3$ & $5.6 \cdot 10^{-3}$ \\
$L_{1,\delta}P_{1}$    & $26$ & $28$ & $8.0 \cdot 10^{-2}$ \\
$\tilde{L}_{2,0}$      & $2$ & $3$ & $1.4 \cdot 10^{-3}$ \\
$\tilde{L}_{2}P_{2}$   & $3$ & $6$ & $1.2 \cdot 10^{-3}$ \\
$P_{2}\tilde{L}_{2}P_{2}$   & $0$ & $6$ & $9.5 \cdot 10^{-5}$ \\
 \hline
\end{tabular}
\end{center}
\caption{Example 5.2: Number of iterations, number of matrix-vector product evaluations
with the matrix $K$, and relative error in approximate solutions ${\bx}_k$ determined by
truncated iteration with RRGMRES using the discrepancy principle and different
regularization matrices for several noise levels.}\label{tab5.2}
\end{table}

\begin{figure*}[h!tbp]
\centering
\subfigure[]{\includegraphics[height=4cm,width=6cm]{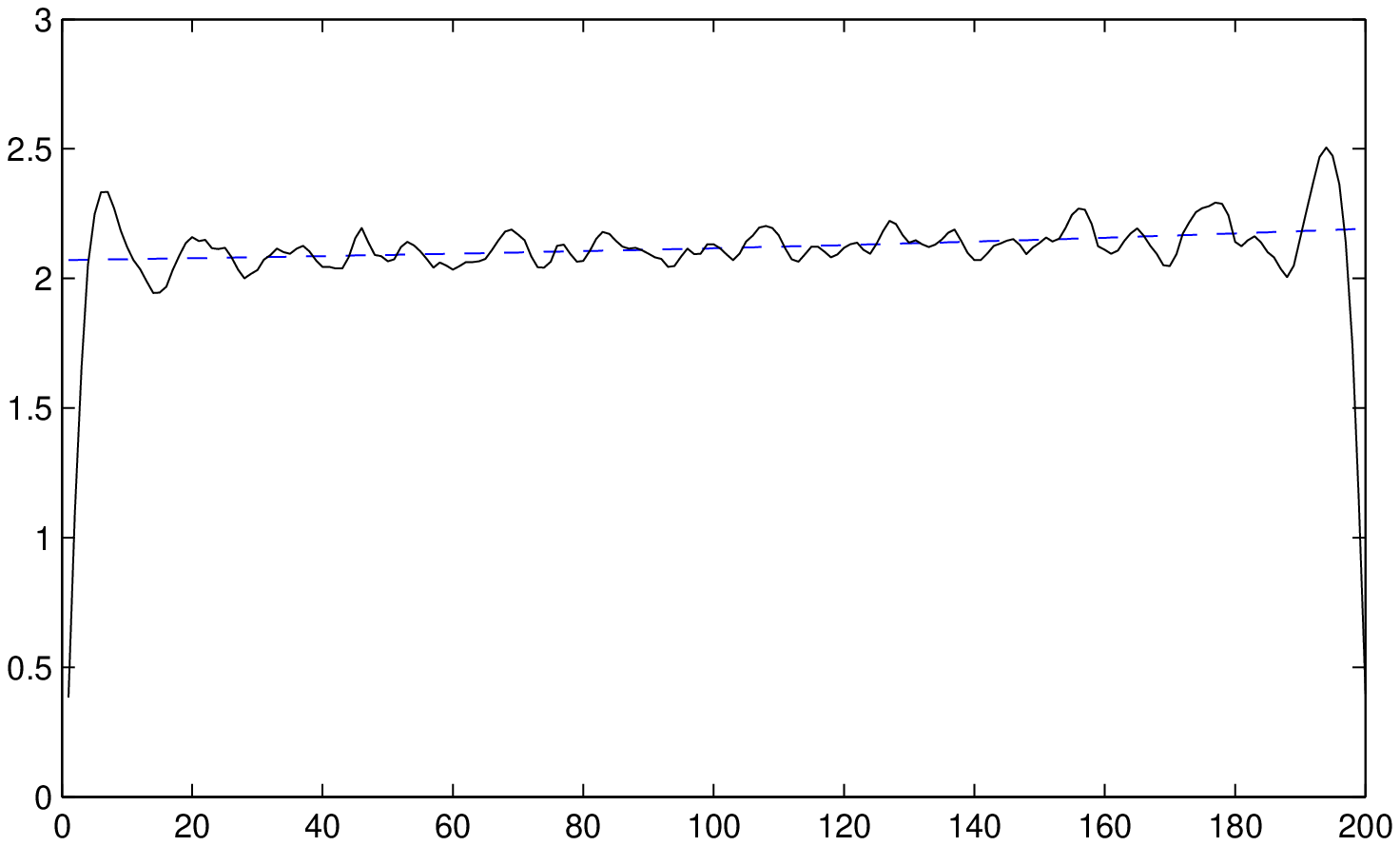}}
\subfigure[]{\includegraphics[height=4cm,width=6cm]{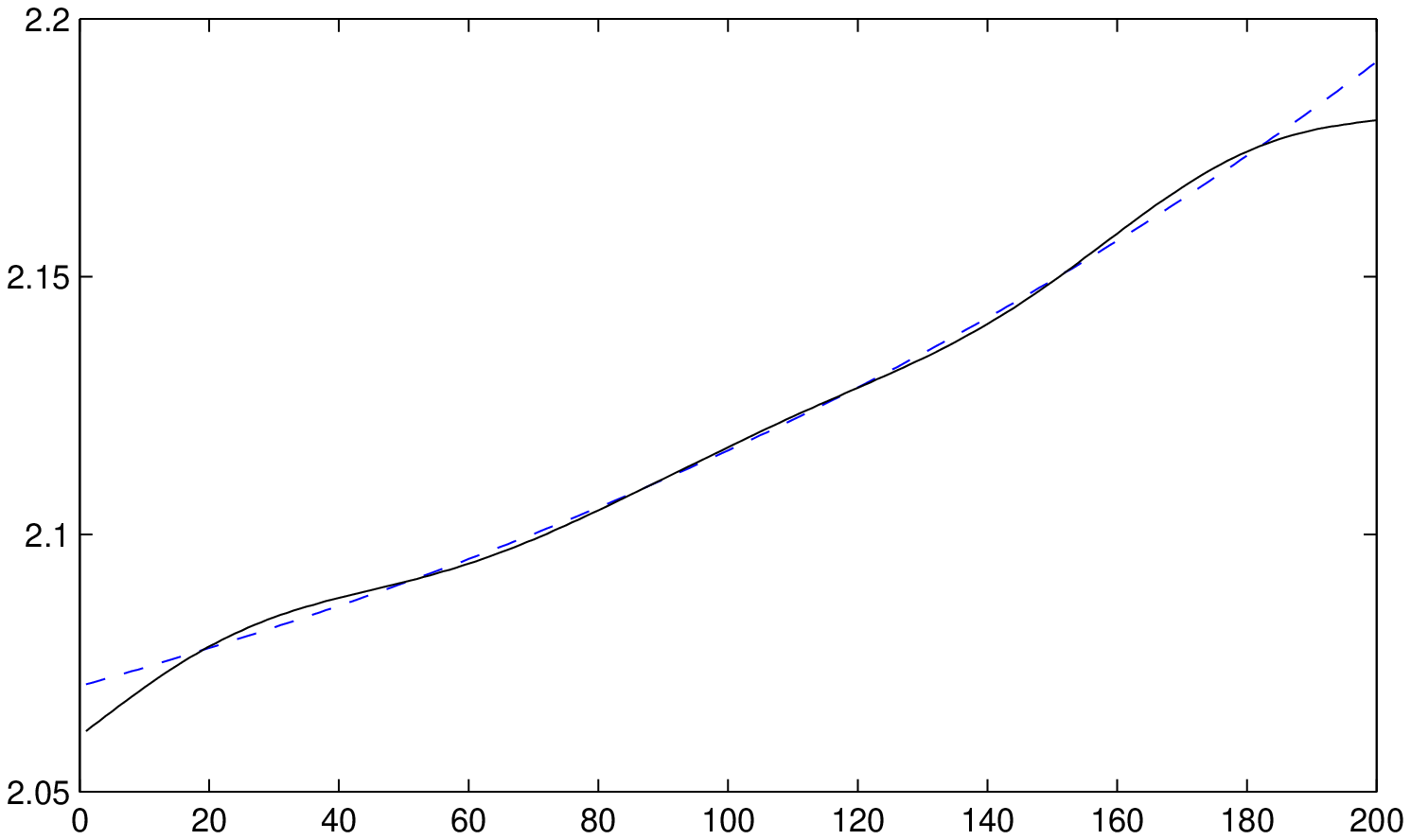}}
\subfigure[]{\includegraphics[height=4cm,width=6cm]{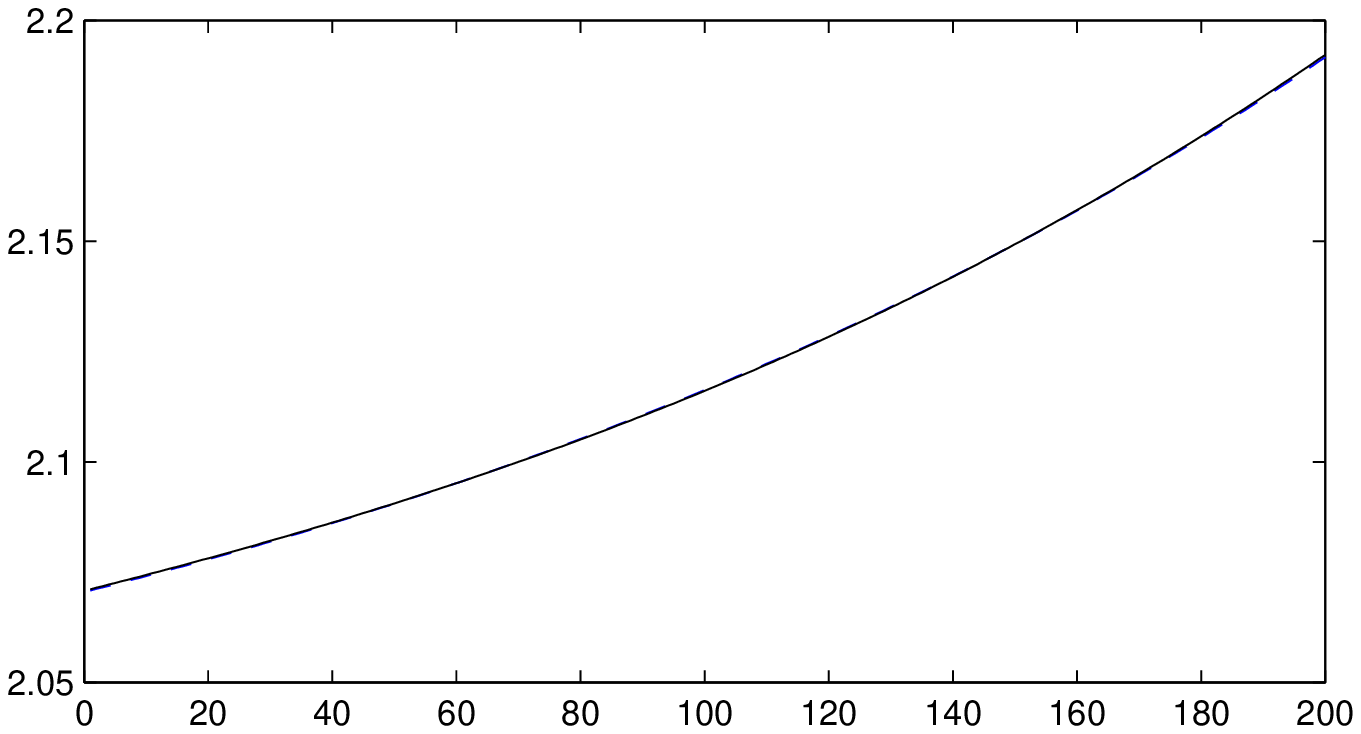}}
\caption{Example 5.2: Continuous curves: Computed approximate solutions ${\bx}_k$
determined by truncated iteration with RRGMRES using the discrepancy principle. The noise
level is $\nu=1\cdot10^{-4}$. (a) Iterate ${\bx}_{13}$ determined without regularization
matrix $(L:=I)$, (b) iterate ${\bx}_{3}$ determined with the regularization matrix
$L=\tilde{L}_{2}P_{2}$ and (c) iterate ${\bx}_{0}$ determined with the regularization matrix
$L=P_{2}\tilde{L}_{2}P_{2}$. The dashed curves show the desired solution $\widehat{\bx}$.}
\label{fig5.2}
\end{figure*}

Example 5.2. Regard the Fredholm integral equation of the first kind,
\begin{equation}\label{deriv2}
\int_0^1 k(s,t)x(t) \, dt = \exp(s)+(1-e)s+1,\qquad 0 \leq s \leq 1,
\end{equation}
where
\[
k(s,t)=\left\{\begin{array}{cc} s(t-1),\qquad &  s< t, \\
   t(s-1),\qquad &  s\geq t.
   \end{array}\right.
\]
We discretize \eqref{deriv2} by a Galerkin method with orthonormal box functions as test
and trial functions using the MATLAB program {\sf deriv2} from \cite{Ha3}. This program
yields a symmetric indefinite matrix $K\in{\R}^{200\times 200}$ and a scaled discrete
approximation of the solution $x(t)=\exp(t)$ of \eqref{deriv2}. Adding
${\bn}_1=[1,1,\ldots,1]^{T}$ yields the vector $\widehat{\bx}\in{\R}^{200}$ with which we
compute the error-free right-hand side $\widehat{\bb}:=K\widehat{\bx}$. Error vectors
$\be\in\R^{200}$ are constructed similarly as in Example 5.1, and the data vector ${\bb}$
in (\ref{eq:sys}) is obtained from (\ref{noisefree}).

Table \ref{tab5.2} shows results obtained with RRGMRES for different regularization matrices.
The performance for three noise levels is displayed. The
iterations are terminated with the aid of the discrepancy principle (\ref{discrp1}).
When $L=\tilde{L}_{2,0}$, $L=\tilde{L}_{2}P_{2}$ or $L=P_{2}\tilde{L}_{2}P_{2}$, and
the noise level is $\nu=1\cdot 10^{-2}$, as well as when $L=P_{2}\tilde{L}_{2}P_{2}$, and
the noise level is $\nu=1\cdot 10^{-3}$ or $\nu=1\cdot 10^{-4}$, the initial residual
${\br}_0:={\bb}-A{\bx}^{(0)}$ satisfies the discrepancy principle and no iterations are
carried out. Figure \ref{fig5.2} shows computed approximate solutions obtained for the
noise level $\nu=1\cdot 10^{-4}$ with the regularization matrix $L=\tilde{L}_{2}P_{2}$ and
without regularization matrix. Table \ref{tab5.2} and Figure \ref{fig5.2} show the
regularization matrix $L=\tilde{L}_{2}P_{2}$ to give the most accurate approximations of
the desired solution $\widehat{\bx}$. We remark that addition of the vector ${\bn}_1$ to
to the solution vector determined by the program  {\sf deriv2} enhances the benefit of
using a regularization matrix different from the identity. The benefit would be even
larger, if a larger multiple of the vector ${\bn}_1$ were added to the solution. $\Box$

The above examples illustrate the performance of regularization matrices suggested by the
theory developed in Section \ref{sec2}. Other combinations of nonsingular regularization
matrices and orthogonal projectors also can be applied. For instance, the regularization
matrix $L=\tilde{L}_2P_1$ performs as well as  $L=\tilde{L}_2P_2$ when applied to the
solution of the problem of Example 5.1.

\section{Conclusion}\label{sec6}
This paper presents a novel method to determine regularization matrices via the solution
of a matrix nearness problem. Numerical examples illustrate the effectiveness of the
regularization matrices so obtained. While all examples used the discrepancy principle to
determine a suitable regularized approximate solution of (\ref{eq:sys}), other parameter
choice rules also can be applied; see, e.g., \cite{Ki,RR} for discussions and references.

\section*{Acknowledgment}
SN is grateful to Paolo Butt\`a for valuable discussions and comments on part of the
present work. The authors would like to thank a referee for comments.

\end{document}